\documentclass[a4paper,11pt, leqno]{amsart}
\usepackage{amsmath,amssymb,cite,mathrsfs}

\newtheorem{theorem}{Theorem}[section]
\newtheorem{lemma}[theorem]{Lemma}
\newtheorem{proposition}[theorem]{Proposition}
\newtheorem{corollary}[theorem]{Corollary}

\theoremstyle{definition}
\newtheorem{definition}[theorem]{Definition}
\newtheorem{example}[theorem]{Example}

\theoremstyle{remark}
\newtheorem{remark}[theorem]{Remark}
\DeclareMathOperator*{\Res}{\mathrm{Res}}
\DeclareMathOperator*{\Lim}{\mathcal{L}\mathit{im}}
\def\laa{{\langle}}
\def\raa{{\rangle}}
\def\ss{{\mathfrak{G}}}

\numberwithin{equation}{section}
\allowdisplaybreaks

\begin{document}

%\hfill\texttt{\jobname.tex}\qquad\today
\title{Infinite series involving hyperbolic functions}

\author{Yasushi Komori}
\address{Y. Komori: Department of Mathematics, Rikkyo University, Nishi-Ikebukuro, Toshima-ku, Tokyo 171-8501, Japan}
\email{komori@rikkyo.ac.jp}

\author{Kohji Matsumoto}
\address{K. Matsumoto: Graduate School of Mathematics, Nagoya University, Chikusa-\
ku, Nagoya 464-8602, Japan}
\email{kohjimat@math.nagoya-u.ac.jp}

\author{Hirofumi Tsumura}
\address{{H.\,Tsumura:} Department of Mathematics and Information Sciences, Tokyo \
Metropolitan University, 1-1, Minami-Ohsawa, Hachioji, Tokyo 192-0397, Japan}
\email{tsumura@tmu.ac.jp}

\keywords{Eisenstein series, Hurwitz numbers, Lemniscate constant, Hurwitz's formula, Barnes multiple zeta-functions, Jacobi theta function, 
Hyperbolic functions, $q$-zeta functions}
\subjclass[2010]{Primary 11M41, Secondary 11B68, 11F27, 11M32, 11M99}

\begin{abstract}
In the former part of this paper, we summarize our previous results on infinite series involving the hyperbolic sine function, especially, with a focus on the hyperbolic sine analogue of Eisenstein series. Those are based on the classical results given by Cauchy, Mellin and Kronecker. In the latter part, we give new formulas for some infinite series involving the hyperbolic cosine function. 
\end{abstract}

\maketitle

\baselineskip 16pt

%%%%%%%%%%%%%%%%%%%%%%%%%%%%%%%%%%%%%%%%%%%%%%%%%%
\section{Introduction}\label{sec-1}
%%%%%%%%%%%%%%%%%%%%%%%%%%%%%%%%%%%%%%%%%%%%%%%%%

The former part of this paper (Sections \ref{sec-2}\,-\,\ref{sec-4}) is a survey of our previous work \cite{KMT-DB,KMT-CM,KMT-CRB,KMT-Forum} on infinite series involving the hyperbolic sine function. In the latter part (Section \ref{sec-5}), we give some new formulas for infinite series involving the hyperbolic cosine function which can be regarded as $\cosh$-versions of the known formulas in \cite{KMT-CM}. 

Let $\mathbb{N}$, $\mathbb{N}_0$, $\mathbb{Z}$, $\mathbb{Q}$, 
$\mathbb{R}$ and $\mathbb{C}$ be the sets of natural numbers, nonnegative integers, rational integers, rational numbers, real numbers and complex numbers, respectively. Let 
$i=\sqrt{-1}=e^{\pi i/2}$, $\rho=e^{2\pi i/3}$, and $\mathbb{H}=\{z\in\mathbb{C}\;|\;\Im z>0\}$. 

We begin with recalling the following classical formula proved by Cauchy \cite{Ca} and also by Mellin \cite{Me1,Me2}, and further, 
rediscovered by Ramanujan (see Berndt \cite{Be2}):
\begin{equation}                                                              
\mathcal{S}_1(4k+3;i)
=\frac{(2\pi)^{4k+3}}{2}\sum_{j=0}^{2k+2}(-1)^{j+1}\frac{B_{2j}(1/2)B_{4k+4-2j}(1/2)}     {(2j)!(4k+4-2j)!}       \label{1}
\end{equation}
for $k\in \mathbb{N}_0$, where 
$$\mathcal{S}_1(s;\tau)=\sum_{m= 1}^\infty\frac{(-1)^m}{\sinh(m\pi i/\tau)m^{s}} \quad (s\in \mathbb{C})$$
for $\tau\in\mathbb{H}$, and 
$\{B_n(x)\}$ are Bernoulli polynomials defined by
\begin{equation*}
\frac{te^{xt}}{e^t-1}=\sum_{n=0}^\infty B_n(x)\frac{t^n}{n!}.
\end{equation*} 

In Section \ref{sec-2}, we state a certain generalization of \eqref{1} which we showed in \cite{KMT-CRB}. As a generalization of $\mathcal{S}_1(s;\tau)$, we consider the series
\begin{align*}
& \sum_{m= 1}^\infty\frac{(-1)^{m}}{\left\{\prod_{j=1}^{n-1}\sinh(m\pi i/\eta^j)\right\} m^s},
\end{align*}
where $n\in \mathbb{Z}_{\geq 2}$ and $\eta=e^{\pi i/n}$, 
%$\prod_\eta$ runs over the all $\eta$ with $\eta^n=-1$ and $\eta\neq -1$, 
and give some formula for this series which includes \eqref{1} (see Theorem \ref{Th-2-1}). 
In fact, the case $n=2$ of that formula implies $\eqref{1}$.

In Section \ref{sec-3}, we consider, for $r\in \mathbb{N}$ and $\tau\in\mathbb{H}$, the double series
\begin{align}
%& \mathcal{G}_k^{[r]}(\tau)=
& \sum_{m\in\mathbb{Z} \atop m\neq 0}\sum_{n\in \mathbb{Z}}
%& \sum_{(m,n)\in \mathbb{Z}^2 \atop m\neq 0}
\frac{(-1)^{rn}}{\sinh(m\pi i/\tau)^r (m+n\tau)^k}\quad (k\in \mathbb{N}).  \label{2}
\end{align}
%and a more general form (see Definition \ref{Def-4-1}).
This generalized form \eqref{2} is necessary for applications
(see, e.g., Proposition \ref{prop_p_zeta}),
but it is to be noted that a delicate convergence problem arises in the case $k=1$.
This double series is regarded as a hyperbolic sine analogue of the classical Eisenstein series
$$G_{2k}(\tau)=
%\sum_{m\in \mathbb{Z}}\sum_{n\in\mathbb{Z} \atop (m,n)\neq (0,0)}
\sum_{(m,n)\in \mathbb{Z}^2 \atop (m,n)\neq (0,0)}
\frac{1}{(m+n\tau)^{2k}}\quad (k\in \mathbb{N}_{\geq 2}).$$
We summarize our research on the series \eqref{2} 
and its further generalizations (see Definition \ref{Def-4-1})
which we showed in \cite{KMT-Forum}, based on the work of Kronecker and Katayama. This includes a previous result given by the third-named author (see \cite{TsBul}). Further we give some applications of this fact to $q$-zeta functions.

In Section \ref{sec-4}, we consider the double series analogue of $S_1(s;\tau)$ defined by
\begin{equation}
\mathcal{S}_2(s;\tau)=
% \sum_{m\in \mathbb{Z}\smallsetminus \{0\}}\sum_{n\in \mathbb{Z} \atop m+n>0}
\sum_{(m,n)\in \mathbb{Z}^2 \atop {m \neq 0 \atop {n\neq 0 \atop m+n>0}}}
\frac{(-1)^{m+n}}{\sinh(m\pi i/\tau)\sinh(n\pi i/\tau)(m+n)^s}\quad (s\in \mathbb{C}). \label{def-S2}
\end{equation}
We state certain generalizations of \eqref{1} which we showed in \cite{KMT-CM}. For example, 
when $\tau=i$, we evaluate $\mathcal{S}_2(-4k;i)$ (see Theorem \ref{Th-3-1}) and $\mathcal{S}_2(4k;i)$ (see Theorem \ref{Th-3-2}) for $k\in \mathbb{N}$.

In Section \ref{sec-5}, we give new formulas for the infinite series involving the hyperbolic cosine function. 
A $\cosh$-version of the result which will be stated in Section \ref{sec-2} was already given in \cite[Remark 6.6]{KMT-CRB}. 
Also the third named-author proved a $\cosh$-version of the result which will be stated in Section \ref{sec-3} (see \cite[Theorem 3.1]{TsResult}).    Therefore the remaining task is to give a $\cosh$-version of the result which will be stated in Section \ref{sec-4}. 
As an analogue of \eqref{1}, we begin by considering Ramanujan's formula 
\begin{equation}
\begin{split}
& \sum_{m=0}^\infty \frac{(-1)^m}{\cosh((m+1/2)\pi)(m+1/2)^{4k+1}} =\frac{{(2\pi)}^{4k+1}}{8}\sum_{j=0}^{2k}(-1)^{j}\frac{E_{2j}(1/2)}{(2j)!}\frac{E_{4k-2j}(1/2)}{(4k-2j)!}
\end{split}
 \label{1-2}
\end{equation}
for $k\in \mathbb{N}_0$ 
(for more general form, see Berndt \cite[p.\,276, Entry 21(ii)]{Be2}), 
where $\{E_n(x)\}$ are Euler polynomials defined by 
$$\frac{2e^{xt}}{e^{t}+1}=\sum_{n=0}^\infty E_n(x)\frac{t^n}{n!}.$$
%Corresponding to this fact, we already gave a $\cosh$-version of the result in Section \ref{sec-2} (see \cite[Remark 6.6]{KMT-CRB}). Also the third named-author proved a $\cosh$-version of the result in Section \ref{sec-4} (see \cite[Theorem 3.1]{TsResult}). Therefore we here aim to give a $\cosh$-version of the result in Section \ref{sec-3}. 
We consider the double series analogous to \eqref{1-2} defined by
\begin{equation}
 \mathcal{C}_2(s;\tau)=
%\sum_{m\in \mathbb{Z}} \sum_{n\in \mathbb{Z} \atop m+n+1>0}
\sum_{(m,n)\in \mathbb{Z}^2 \atop m+n+1>0}
\frac{(-1)^{m+n}}{\cosh((m+1/2)\pi i/\tau)\cosh((n+1/2)\pi i/\tau)(m+n+1)^{s}}
 \label{1-3}
\end{equation}
for $s\in \mathbb{C}$ and $\tau\in \mathbb{H}$. 
This series can also be regarded as the hyperbolic cosine analogue of \eqref{def-S2}.
It should be noted that 
we can prove a functional relation between $\mathcal{C}_2(s;\tau)$ and $S_1(s;\tau)$ (see Theorem \ref{Th-main-1}), though we cannot obtain the corresponding result on $\mathcal{S}_2(s;\tau)$. 
From this fact we can deduce several arithmetic consequences on special values of $\mathcal{C}_2(s;\tau)$; for instance,
putting $\tau=i$ and using Theorem \ref{Th-1}, we obtain 
$$\mathcal{C}_2(-4k+2;i)\in \mathbb{Q}\cdot \left(\frac{\varpi}{\pi}\right)^{4k}\qquad (k\in \mathbb{N})$$
(see Corollary \ref{Cor-4-3} and Example \ref{exam-2}),
where $\varpi$ is
the ``lemniscate constant'' defined by
$$
\varpi=2\int_0^1\frac{dx}{\sqrt{1-x^4}}=\frac{\Gamma(1/4)^2}{2\sqrt{2\pi}}
=2.62205\cdots.
$$
 We also evaluate $\mathcal{C}_2(-6k+2;\rho)$ for $k\in \mathbb{N}$ using the result in Theorem \ref{Th-2} (see Example \ref{exam-3}). Furthermore we can obtain 
$$\mathcal{C}_2(4k;i)\in \mathbb{Q}\cdot \pi^{4k-1} \qquad (k\in \mathbb{N}),$$
which can be regarded as a double analogue of \eqref{1-2} (see Corollary \ref{C-1} and Example \ref{Exam-1}).
\bigskip

A part of the present article, mainly Section \ref{sec-3}, is the written version of the second author's talk at the 11th Vilnius Conference on Probability Theory and Mathematical Statistics.
On this occasion the second author expresses his sincere gratitude to the organizers, especially Professor E. Manstavi{\v c}ius, for the invitation and the kind hospitality.

\ 

%%%%%%%%%%%%%%%%%%%%%%%%%%%%%%%%%%%%%%%%%%%%
\section{A generalization of the Cauchy-Mellin formula}\label{sec-2}
%%%%%%%%%%%%%%%%%%%%%%%%%%%%%%%%%%%%%%%%%%%%

In our previous paper \cite{KMT-CRB}, we give the following generalization of the 
Cauchy-Mellin formula \eqref{1}:

\begin{theorem}[\cite{KMT-CRB}\ Theorem 6.1]\label{Th-2-1}
Let $n \in \mathbb{Z}_{\geq 2}$ and $\eta=e^{\pi i/n}$.
Assume $0<y<1$ and $p \in \mathbb{Z}$,
or $y=0,1$ and $p>n/2$.
We have
\begin{multline}\label{2-1}
  \mathcal{C}_{2p+1}\sum_{m\in \mathbb{Z} \atop m\neq 0}
\frac{\cos(2m\pi y)}{m^{2p+1-n}} \left(
\prod_{j=1}^{n-1}
\frac{\cosh\left( 2m\pi i\eta^{j}(y-1/2)\right)}{\sinh\left( m\pi i\eta^{j}\right)}\right)
 \\
 = -\frac{2^{n-1}(2\pi i)^{2p+1-n}}{\eta^{n(n-1)/2}}
\sum_{m_1,\ldots,m_n = 0 \atop m_1+\cdots+m_n=p}^\infty \prod_{\nu=1}^{n} \frac{B_{2m_\nu}(y)}{(2m_\nu) !}\eta^{2(\nu-1)m_\nu},
\end{multline}
where 
\begin{equation*}
\mathcal{C}_h= \sum_{j=0}^{n-1}\eta^{j(1-h)}=
\begin{cases}
n & \text{if $h\equiv 1$ \ {\rm (mod $2n$)}}, \\
0 & \text{if $h\not\equiv 1$ \ {\rm (mod $2n$)} \ and $2\nmid h$},\\
\frac{2}{1-\eta^{1-h}} & \textrm{if $h\not\equiv 1$ \ {\rm (mod $2n$)} \ and $2\mid h$}.
\end{cases}
\end{equation*}
\end{theorem}

\begin{example}
It is easy to confirm that the case $(n,y)=(2,1/2)$ implies \eqref{1}. Furthermore it follows from \eqref{2-1} that
\begin{align*} 
& \sum_{m=1}^\infty\frac{1}{\sinh(m\pi i/\rho)^2 m^4} =-\frac{1}{5670}\pi^4, \\
& \sum_{m=1}^\infty\frac{(-1)^m}{\sinh(m\pi i\xi)\sinh(m\pi i\xi^2)\sinh(m\pi i\xi^{3})\sinh(m\pi i\xi^4)m^6} =-\frac{1}{93550}\pi^6, \\
& \sum_{m=1}^\infty\frac{\coth(m\pi i/\rho)^2 }{m^{10}} =\frac{40247}{3831077250}\pi^{10}, 
\end{align*}
where 
$\xi=e^{2\pi i/8}$.
\end{example}

The formula \eqref{2-1} is deduced from the functional equation for  Barnes multiple zeta-functions (see \cite[Theorem 2.1]{KMT-CRB}, also proved independently by Shibukawa \cite{Shibukawa}) which we will again consider in Section \ref{sec-4}. In the case of the double zeta-function, that functional equation can be read as 
\begin{align}
  \zeta_2(s;y;1,\tau):&=\sum_{m=0}^\infty\sum_{n=0}^\infty 
   \frac{1}{((1-y)(1+\tau)+m+n\tau)^s} \label{B-FEN}\\
  & =-
  \frac{2\pi i}{\Gamma(s)(e^{2\pi is}-1)} 
  \bigg\{\sum_{m\in \mathbb{Z} \atop m\neq 0} 
  \Bigl( 
  \frac{e^{(2m\pi i\tau)y}}
  {e^{2m\pi i\tau}-1}
  \Bigr) (2m\pi i)^{s-1}e^{2m\pi i y}\notag\\
&\qquad + \frac{1}{\tau} 
  \sum_{m\in \mathbb{Z} \atop m\neq 0}
%  \sum_{m\in\mathbb{Z}\smallsetminus\{0\}} 
  \Bigl( 
  \frac{e^{(2m\pi i/\tau)y}}
  {e^{2m\pi i/\tau}-1}
  \Bigr) (2m\pi i/\tau)^{s-1}e^{2m\pi i y}\bigg\}\notag
\end{align}
for $y\in [0,1)$ and $\tau\in \mathbb{H}$. 
In particular when $y=1/2$ and $\tau=i$, we have
\begin{align}\label{zeta_2-S_1}
  \zeta_2(s;1/2;1,i):&=\sum_{m=0}^\infty\sum_{n=0}^\infty  
   \frac{1}{(1/2+m+(1/2+n)i)^s}\\
  & =-
  \frac{(2\pi)^s(e^{\pi is/2}-1)}{\Gamma(s)(e^{2\pi is}-1)} S_1(1-s;i).\notag
  %\sum_{m\in\mathbb{Z}\smallsetminus\{0\}} 
  %\frac{(-1)^m}{\sinh(m\pi)m^{1-s}}.
\end{align}
Therefore we may say that $S_1(s;i)$ is the "dual" of the Barnes double zeta-function via the functional equation \eqref{zeta_2-S_1}.
Setting $s=-4k-2$ in \eqref{zeta_2-S_1}, we obtain \eqref{1}. More generally, using the functional equation for Barnes multiple zeta-functions and letting $s$ be a suitable nonpositive integer, we can obtain \eqref{2-1}.

In Section \ref{sec-4}, we consider the values of $S_1(s;\tau)$ at nonpositive integers. 

\ 

%%%%%%%%%%%%%%%%%%%%%%%%%%%%%%%%%%%%%%%%%%%%%%%%%%%%%%%
\section{Hyperbolic sine analogues of Eisenstein series}\label{sec-3}
%%%%%%%%%%%%%%%%%%%%%%%%%%%%%%%%%%%%%%%%%%%%%%%%%%%%%%

In this section, we present the results on hyperbolic sine analogues of Eisenstein series given in \cite{KMT-DB,KMT-Forum}.

For $\tau\in \mathbb{H}$, we consider the Eisenstein series and the level $2$ Eisenstein series defined by 
\begin{align*}
& G_{2k}(\tau):=
%\sum_{m \in \mathbb{Z}}\sum_{n \in \mathbb{Z} \atop (m,n)\neq (0,0)}
\sum_{(m,n)\in \mathbb{Z}^2 \atop (m,n)\neq (0,0)}
\frac{1}{(m+n\tau)^{2k}}\qquad (k\in \mathbb{N}_{\geq 2}),\\
& G_{2k}^{{\bf 1}}(\tau)=G_{2k}(\tau;(1,1);2):=
%\sum_{m \in \mathbb{Z}}\sum_{n \in \mathbb{Z}}
\sum_{(m,n)\in \mathbb{Z}^2}
\frac{1}{(2m+1+(2n+1)\tau)^{2k}}\qquad (k\in \mathbb{N}_{\geq 2})
\end{align*}
(see Hecke \cite{Hecke1927}, also Koblitz \cite{Ko} and Serre \cite{Se}). 
%The values of $\zeta_2(s;1/2;1,i)$ at positive integers are closely related 
%to the following well-known results, namely %
We recall the following classical formula of Hurwitz \cite{Hur99} and that of Katayama \cite{Katayama}:
\begin{align*}
& G_{4k}(i)=\frac{(2\varpi)^{4k}}{(4k)!}H_{4k}\qquad (k\in \mathbb{N}),\\
& G_{4k}^{{\bf 1}}(i)=\frac{(2\varpi)^{4k}}{(4k)!}H_{4k}^{{\bf 1}}\qquad (k\in \mathbb{N}),
\end{align*}
where 
$H_{4k}$ and $H_{4k}^{{\bf 1}}$ are the Hurwitz number and the $2$-division Hurwitz number, respectively. Actually $H_{4k}$ is defined by the Laurent expansion coefficient
of the Weierstrass $\wp$-function:
\begin{align*}
\wp(z)&=\frac{1}{z^2}+\sum_{\lambda}\left(\frac{1}{(z-\lambda)^2}-\frac{1}{\lambda^2}
\right)=\frac{1}{z^2}+\sum_{k=1}^{\infty}\frac{2^k H_k}{k}\frac{z^{k-2}}{(k-2)!},
\end{align*}
where $\lambda$ runs over all non-zero lattice points spanned by $\varpi$ and
$\varpi i$, that is, $\lambda=m\varpi+n\varpi i$ ($m,n\in \mathbb{Z}$), and $H_{4k}^{{\bf 1}}$ can be similarly defined (see \cite[Section 6]{Katayama}).

As hyperbolic sine analogues of these results, the third-named author \cite{TsBul,TsResult} considered the series
\begin{align*}
%& \sum_{m\in\mathbb{Z}\smallsetminus\{0\}}\sum_{n\in\mathbb{Z}}
%& \sum_{(m,n)\in \mathbb{Z}^2 \atop m\neq 0}
& \sum_{m\in\mathbb{Z} \atop m\neq 0}\sum_{n\in\mathbb{Z}}
\frac{(-1)^n}{\sinh(m\pi)(m+ni)^k}\qquad (k\in \mathbb{N}),\\
%& \sum_{(m,n)\in \mathbb{Z}^2}
%\sum_{m\in \mathbb{Z}}\sum_{n\in\mathbb{Z}}
& \sum_{m\in\mathbb{Z}}\sum_{n\in\mathbb{Z}}
\frac{(-1)^n}{\sinh((2m+1)\pi)(2m+1+(2n+1)i)^k} \qquad (k\in \mathbb{N}),
\end{align*}
and so on. For example, he gave
\begin{align}
%& \sum_{m\in\mathbb{Z}\smallsetminus\{0\}}\sum_{n\in\mathbb{Z}}
& \sum_{(m,n)\in \mathbb{Z}^2 \atop m\neq 0}
\frac{(-1)^n}{\sinh(m\pi)(m+ni)^3}
=\frac{\varpi^4}{15\pi}-\frac{7}{90}\pi^3+\frac{1}{6}\pi^2, \label{2-3}\\
%&\sum_{m\in \mathbb{Z}}\sum_{n\in\mathbb{Z}}
& \sum_{(m,n)\in \mathbb{Z}^2}
\frac{(-1)^n}{\sinh((2m+1)\pi)(2m+1+(2n+1)i)^6}
=\frac{i\pi^2}{1536}(4\varpi^4-5\pi^3).\notag
\end{align}
The method to prove them is to use the property of the absolutely and uniformly convergent associated double series, which is called the `$u$-method' (see, for example, \cite{MT2006,TsAust}).  

Here it is to be stressed that these results can be deduced from a more general formula as follows.
Let $\tau\in\mathbb{H}$.
Define the two-variable Eisenstein series
$$
\mathcal{Z}_2(s_1,s_2;\tau)=
\sum_{m\in\mathbb{Z} \atop m\neq 0}\sum_{n\in\mathbb{Z}}
%\sum_{(m,n)\in \mathbb{Z}^2 \atop m\neq 0}
\frac{1}{m^{s_1}(m+n\tau)^{s_2}}
$$
and its hyperbolic sine analogue
$$
\ss_2(s_1,s_2;\tau)=
%\sum_{m\in\mathbb{Z}\smallsetminus\{0\}}\sum_{n\in\mathbb{Z}}
%\sum_{(m,n)\in \mathbb{Z}^2 \atop m\neq 0}
\sum_{m\in\mathbb{Z} \atop m\neq 0}\sum_{n\in\mathbb{Z}}
\frac{(-1)^n}
{\sinh(m\pi i/\tau)m^{s_1}(m+n\tau)^{s_2}}.
$$
\bigskip
Then we have the following theorem.

\begin{theorem}[\cite{KMT-DB}\ Theorem 3.2]\label{Th-3-1}
For $k\in\mathbb{Z}_{\geq 2}$, $s\in\mathbb{C}$ and $\tau\in\mathbb{H}$, we have
\begin{equation}
\ss_2(s,k;\tau)=\frac{\tau}{\pi i}\sum_{\nu=0}^{[k/2]}\frac{(2\pi i/\tau)^{2\nu}}{(2\nu)!}
B_{2\nu}(1/2)\mathcal{Z}_2(s,k+1-2\nu;\tau).  \label{2-4}
\end{equation}
\end{theorem}

%We see that for $j\in \mathbb{N}$, 
%$$\mathcal{Z}_2(0,2j;\tau)=G_{2j}(\tau)-2\tau^{-2j}\zeta(2j),$$
%where $\zeta(s)$ is the Riemann zeta-function and 
We define 
$$G_2(\tau)=
\sum_{m \in \mathbb{Z}}\sum_{n \in \mathbb{Z} \atop (m,n)\neq (0,0)}
%\sum_{(m,n)\in \mathbb{Z}^2 \atop (m,n)\neq (0,0)}
\frac{1}{(m+n\tau)^{2}}.$$
It is known that $G_2(i)=-\pi$ (see \cite[Section 7]{Se}). 

\begin{example}
Setting $(k,\tau)=(3,i)$ in \eqref{2-4}, we can recover \eqref{2-3}, that is,
\begin{align*}
\ss_2(0,3;i)& =\frac{1}{\pi}\left\{ G_4(i)-2\zeta(4)+\frac{(2\pi)^{2}}{2!}B_{2}(1/2)\left( G_{2}(i)+2\zeta(2)\right)\right\}\\
 & =\frac{\varpi^4}{15\pi}-\frac{7}{90}\pi^3+\frac{1}{6}\pi^2 
\end{align*}
from $G_4(i)=\varpi^4/15$, $B_2(1/2)=-1/12$, and 
$$\mathcal{Z}_2(0,2j;i)=G_{2j}(i)-2(-1)^j\zeta(2j)\quad (j\in \mathbb{N}), $$
where $\zeta(s)$ is the Riemann zeta-function (see \cite[(36),\ (37)]{KMT-DB}). 
\end{example}

Inspired by these results, 
we introduce a general form of the hyperbolic sine analogue of Eisenstein series which we studied in \cite{KMT-Forum}.

First we recall the classical work of Kronecker \cite{Kro} and the work of Katayama \cite{Katayama}. 
Define the Jacobi odd theta function by
\begin{align}
\label{theta-def}
\theta(z,\tau)=-i\sum_{n\in\mathbb{Z}}\exp\left(\pi i(n+\frac{1}{2})^2\tau
+2\pi i(n+\frac{1}{2})z+\pi in\right)
\end{align}
and put
\begin{align*}
\mathcal{E}(\xi,x,y,\omega_1,\omega_2)
=\frac{e^{2\pi ix\xi/\omega_1}}{\omega_1}\frac{\theta^{\prime}(0,\tau)
\theta(\xi/\omega_1+x\tau-y,\tau)}{\theta(\xi/\omega_1,\tau)\theta(x\tau-y,\tau)},
\end{align*}
where $\xi\in\mathbb{C}$, $x,y\in\mathbb{R}$, $0<x<1$,
$\omega_1,\omega_2\in\mathbb{C}$, $\tau=\omega_2/\omega_1\in\mathbb{H}$.
In his study of elliptic functions, Kronecker proved
$$
\mathcal{E}(\xi,x,y,\omega_1,\omega_2)=\lim_{M\to\infty}\lim_{N\to\infty}
\sum_{\substack{-M\leq m\leq M \\ -N\leq n\leq N}}\frac{e^{-2\pi i(mx+ny)}}
{\xi+m\omega_1+n\omega_2},
$$
and also gave the 
Laurent expansion:
$$
\mathcal{E}(\xi,x,y,\omega_1,\omega_2)=\frac{1}{\xi}+\sum_{j=0}^{\infty}
\frac{\mathcal{H}_{j+1}(x,y,\omega_1,\omega_2)}{(j+1)!}\xi^j.
$$
Later, Katayama proved
\begin{align}
&\lim_{M\to\infty}\lim_{N\to\infty}
\sum_{\substack{-M\leq m\leq M \\ -N\leq n\leq N \\ (m,n)\neq(0,0)}}
\frac{e^{-2\pi i(mx+ny)}}{(m\omega_1+n\omega_2)^{j+1}}=\frac{(-1)^j}{(j+1)!}\mathcal{H}_{j+1}(x,y,\omega_1,\omega_2). \label{Kat}
\end{align}
%$$
%\left({\rm cf.\; Hurwitz\; formula}\quad 
%{\sum_{m,n}}^{\prime}\frac{1}{(m+ni)^k}=\frac{(2\varpi)^k}{k!}H_k \right)
%$$
Since this gives a kind of generalization of Hurwitz's formula,
we call $\mathcal{H}_{j+1}(x,y,\omega_1,\omega_2)$ the generalized Hurwitz number. 
Note that the case $(x,y)=(0,0)$ of Katayama's formula \eqref{Kat} was already obtained by
Herglotz (see \cite{Her}). 

For our present aim, it is necessary to show a hyperbolic sine analogue of \eqref{Kat}. Let
$$
\mathcal{F}(\xi,z,\omega_2)=\frac{2\pi i}{\omega_2}\frac{e^{2\pi i\xi z/\omega_2}}
{e^{2\pi i\xi/\omega_2}-1},
$$
which is essentially the generating function of Bernoulli polynomials,
and define
$$
\mathcal{D}_r(\xi)=\mathcal{E}(\xi,x,y,\omega_1,\omega_2)
\mathcal{F}(\xi,z,\omega_2)^r\qquad(r\in\mathbb{N}).
$$
Further, let
$\mathfrak{K}_r(\xi)$ be the residue at $\eta=0$ of the function
$$
\frac{\mathcal{D}_r(\xi)}{\eta}-\mathcal{D}_r(\eta)\mathcal{F}(\xi-\eta,\{y+rz\},
\omega_2).
$$
Then we can show that, for $-1<x,y<1$, $0\leq z\leq 1$, the function 
$\mathfrak{K}_r(\xi)$ is meromorphic in $\xi$, and especially 
{\it holomorphic} at $\xi=0$.   Therefore we can expand $\mathfrak{K}_r(\xi)$
at $\xi=0$:
$$
\mathfrak{K}_r(\xi)=\sum_{k=1}^{\infty}\frac{\mathcal{K}_{k,r}(x,y,z,\omega_1,
\omega_2)}{k!}\xi^{k-1}.
$$
This $\mathcal{K}_{k,r}(x,y,z,\omega_1,\omega_2)$ plays a role similar to that of
generalized Hurwitz numbers in our present situation.

Corresponding to this result, we define a general form of hyperbolic sine analogue of Eisenstein series.

\begin{definition}[\cite{KMT-Forum} Section 3]\label{Def-4-1}
\begin{equation}                                                                        
\label{Gene-Hur}                                                                        
\begin{split}                    
  & \mathcal{G}_{k}^{\langle r\rangle}(x,y,z;\omega_1,\omega_2) \\
  &  =
    \begin{cases}
      & \displaystyle{
      {\Lim}_{M,N}
      \sum_{\substack{-M\leq m\leq M\\-N\leq n\leq N\\ m\neq 0}}
      \frac{(-1)^{rn}}{(\sinh(m\pi i/\tau))^r}
      \frac{e^{2\pi i(m(x+r(z-1/2)/\tau)+n(y+r(z-1/2)))}}{(m\omega_1+n\omega_2)^{k}}
      \quad}\\
     &  \qquad \qquad \Biggl(\begin{aligned}
        &\text{for }k=1\text{ and }0<z<1,y+rz\not\in\mathbb{Z}, \text{ or }\\
        &k=2, (x,y)\neq(0,0)\text{ and }z=0,1
         \end{aligned}\Biggr),\\
      & \displaystyle{
      \sum_{(m,n)\in \mathbb{Z}^2 \atop m\neq 0}
      %\sum_{m \in \mathbb{Z}\smallsetminus\{0\}}\sum_{n \in \mathbb{Z}}
      \frac{(-1)^{rn}}{(\sinh(m\pi i/\tau))^r}
      \frac{e^{2\pi i(m(x+r(z-1/2)/\tau)+n(y+r(z-1/2)))}}{(m\omega_1+n\omega_2)^{k}}
      \quad }\\
      & \qquad \qquad \Biggl(\begin{aligned}
        &\text{for }k\geq 3,\text{ or } \\
        &k=2 \text{ and } 0<z<1
      \end{aligned}\Biggr),
    \end{cases}
\end{split}
\end{equation}
where we denote by $\Lim_{M,N}$
any one of the following limits:
\begin{equation}                                                                        
\label{eq:limit}                                                                        
  \lim_{\substack{M\to\infty\\N\to\infty}},\qquad                                       
  \lim_{M\to\infty}\lim_{N\to\infty},\qquad                                             
  \lim_{N\to\infty}\lim_{M\to\infty}.                                                   
\end{equation}
Here the first limit means that $M=\{M_k\}_{k=1}^{\infty}$, $N=\{N_k\}_{k=1}^{\infty}$,
such that for any $R>0$ one can find $K=K(R)$ for which $M_k\geq R$, $N_k\geq R$ hold
for any $k\geq K$.
\end{definition}

Then, as a hyperbolic sine analogue of Katayama's formula \eqref{Kat}, we obtain

\begin{theorem}[\cite{KMT-Forum} Theorem 3.2]\label{Th-4-1}
With the above notation, 
\begin{align*}
&\mathcal{G}_{k}^{\langle r\rangle}(x,y,z;\omega_1,\omega_2)
%&\lim_{M\to\infty}\lim_{N\to\infty}
%\sum_{\substack{-M\leq m\leq M \\ -N\leq n\leq N \\ m\neq 0}}
%\frac{(-1)^{rn}}{(\sinh(m\pi i/\tau))^r}\frac{e^{2\pi i\{m(x+r(z-1/2)/\tau)+n(y%+r(z-1/2))\}}}
%{(m\omega_1+n\omega_2)^{k}}\\                                
=-\frac{1}{k!}\left(\frac{\omega_2}{\pi i}\right)^r
\mathcal{K}_{k,r}(x,y,z,\omega_1,\omega_2).
\end{align*}
\end{theorem}
\bigskip

%Remark: When $r=1$ and $z=1/2$, the above is essentially Katayama's formula.

%\begin{proof}
We briefly sketch the proof of this theorem. 
Let $C_{M,N}$ be the parallelogram, centered at the origin, and on which are the
points $(M+1/2)\omega_1$, $(N+1/2)\omega_2$.   

First, we show that
$$
{\Lim}_{M,N}\int_{C_{M,N}}\xi^{-k}\mathcal{E}(\xi)\mathcal{F}(\xi)^r d\xi=0
$$
by careful estimation of the integrand, with noting the quasi-periodicity of
$\mathcal{E}$ and $\mathcal{F}$, such as
$$
\mathcal{E}(\xi+\omega_1,x,y,\omega_1,\omega_2)=
\mathcal{E}(\xi,x,y,\omega_1,\omega_2)e^{2\pi ix}
$$
(recall that $\mathcal{E}$ is defined by using the theta function), and
$$
\mathcal{F}(\xi+\omega_2)=\mathcal{F}(\xi)e^{2\pi iz}.
$$
On the other hand, by residue calculus, the integral is written as a sum of residues of
$\xi^{-k}\mathcal{E}(\xi)\mathcal{F}(\xi)^r$ at $\xi=m\omega_1+n\omega_2$.
Finally, each residue can be evaluated by using some properties of Bernoulli polynomials. 
Combining these results, we obtain the assertion of Theorem \ref{Th-4-1}.
%\end{proof}
\bigskip

Theorem \ref{Th-4-1} implies that the hyperbolic sine analogue of Eisenstein series can be written in terms of 
$\mathcal{K}_{k,r}(x,y,z,\omega_1,\omega_2)$.
To evaluate $\mathcal{K}_{k,r}(x,y,z,\omega_1,\omega_2)$, we need the following 
definitions. 

\begin{definition}
The Bernoulli polynomial of higher order $B_k^{\langle r\rangle}(x)$ is defined by
$$
\left(\frac{te^{xt}}{e^t-1}\right)^r=\sum_{k=0}^{\infty}B_k^{\langle r\rangle}(x)
\frac{t^k}{k!}.
$$
\end{definition}

This definition of the Bernoulli polynomial of higher order is slightly
modified from the original definition of N{\"o}rlund (see \cite{Nor0,Nor}).

\begin{definition}
We define $\mathcal{B}^{\laa r \raa}_k(z;\omega_2)$ by
$$
\mathcal{F}(\xi,z,\omega_2)^r=\sum_{k=0}^{\infty}\frac
{\mathcal{B}^{\laa r \raa}_k(z;\omega_2)}{k!}\xi^{k-r}.
$$
\end{definition}

Recall that $\mathcal{K}_{k,r}(x,y,z,\omega_1,\omega_2)$ are Taylor coefficients of
\begin{align}\label{hikaku}
\mathfrak{K}_r(\xi)={\rm Res}_{\eta=0}\left(
\frac{\mathcal{D}_r(\xi)}{\eta}-\mathcal{D}_r(\eta)\mathcal{F}(\xi-\eta,\{y+rz\},       
\omega_2)\right). 
\end{align}
The right-hand side consists of $\mathcal{D}_r$ and $\mathcal{F}$, and
further, $\mathcal{D}_r$ consists of $\mathcal{E}$ and $\mathcal{F}^r$.
%$$                                                                                      
%\mathcal{D}_r(\xi)=\mathcal{E}(\xi,x,y,\omega_1,\omega_2)                               
%\mathcal{F}(\xi,z,\omega_2)^r.                                                          
%$$
Since $\mathcal{H}_k$ and $\mathcal{B}_k^{\laa r\raa}$ are Laurent coefficients of
$\mathcal{E}$ and $\mathcal{F}^r$, respectively, 
comparing the coefficients of the both sides of \eqref{hikaku}, we can find a relation
among $\mathcal{K}_{k,r}$, $\mathcal{H}_k$ and $B_k^{\langle r\rangle}$.

\begin{theorem}[\cite{KMT-Forum}\ Theorem 3.4]\label{Th-4-2}
Assume $0\leq z\leq 1$ and let $k,r\in\mathbb{N}$.
For $(x,y)\neq(0,0)$,
  \begin{multline}                                                                      
    \label{K_kq-3}                                                                      
    \mathcal{K}_{k,r}(x,y,z;\omega_1,\omega_2)                                          
    \\                                                                                  
    \begin{aligned}                                                                     
      &=                                                                                
      k!                                                                                
      \sum_{l=r+1}^{k+r}                                                                
      \frac{\mathcal{H}_l(x,y;\omega_1,\omega_2)}{l!}                                   
      \frac{\mathcal{B}^{\laa r \raa}_{k+r-l}(z;\omega_2)}{(k+r-l)!}                    
      \\                                                                                
      &\quad+                                                                           
      k!                                                                                
      \sum_{l=0}^{r}                                                                    
      \frac{\mathcal{H}_l(x,y;\omega_1,\omega_2)}{l!}\\                                 
      &\quad \qquad \times \Bigl(                                                       
      \frac{\mathcal{B}^{\laa r \raa}_{k+r-l}(z;\omega_2)}{(k+r-l)!}                    
      -                                                                                 
      \sum_{j=0}^{r-l}                                                                  
      \frac{\mathcal{B}^{\laa r \raa}_{r-j-l}(z;\omega_2)}{(r-j-l)!}                    
      \frac{(-1)^j}{j!}                                                                 
      \frac{(k+j-1)!}{(k-1)!}\frac{\mathcal{B}_{k+j}(\{y+rz\};\omega_2)}{(k+j)!}        
      \Bigr).         
    \end{aligned}
  \end{multline}
For $(x,y)=(0,0)$ and $z\neq 1$,
\begin{multline}                                                                        
\label{K_kq-2}                                                                          
    \mathcal{K}_{k,r}(0,0,z;\omega_1,\omega_2)                                          
    \\                                                                                  
    \begin{aligned}                                                                     
      &=                                                                                
      k!                                                                                
      \sum_{l=r+1}^{k+r}                                                                
      \frac{\mathcal{H}_l(\omega_1,\omega_2)}{l!}                                       
      \frac{\mathcal{B}^{\laa r \raa}_{k+r-l}(z;\omega_2)}{(k+r-l)!}                    
      \\                                                                                
      &\qquad+                                                                          
      k!                                                                                
      \sum_{\substack{l=0\\l\neq 1}}^{r}                                                
      \frac{\mathcal{H}_l(\omega_1,\omega_2)}{l!}\\                                     
      &\qquad \qquad \times                                                             
      \Bigl(                                                                            
      \frac{\mathcal{B}^{\laa r \raa}_{k+r-l}(z;\omega_2)}{(k+r-l)!}                    
      -                                                                                 
      \sum_{j=0}^{r-l}                                                                  
      \frac{\mathcal{B}^{\laa r \raa}_{r-j-l}(z;\omega_2)}{(r-j-l)!}                    
      \frac{(-1)^j}{j!}                                                                 
      \frac{(k+j-1)!}{(k-1)!}\frac{\mathcal{B}_{k+j}(\{rz\};\omega_2)}{(k+j)!}          
      \Bigr).                                                                           
    \end{aligned}                                                                       
  \end{multline}
\end{theorem}

In the case $(x,y,\omega_1,\omega_2)=(0,0,1,i)$, we obtain the following result,
because 
$\mathcal{B}^{\laa r \raa}_k(z;\omega_2)=(2\pi i/\omega_2)^k B_k^{\laa r \raa}(z)$.

\begin{theorem}[\cite{KMT-Forum} Corollary 3.5]\label{Th-4-3}
Let $r\in\mathbb{N}$. 
Assume that $k\geq3$ and $0\leq z\leq 1$, or that $k=2$ and $0<z<1$,
or that $k=1$, $0<z<1$ and $rz\not\in\mathbb{Z}$.
Then we have
\begin{align}
&    \pi^r \mathcal{G}_{k}^{\langle r\rangle}(0,0,z;1,i)    \label{eq:formula_0}
\\
   &=\sum_{l=r+1}^{k+r}
      \frac{(2\varpi)^lH_l}{l!}
      \frac{(2\pi)^{k+r-l}B^{\laa r \raa}_{k+r-l}(z)}{(k+r-l)!} \notag\\
    &  +
      \sum_{l=4}^{r}
      \frac{(2\varpi)^lH_l}{l!}(2\pi)^{k+r-l}\Bigl(
      \frac{B^{\laa r \raa}_{k+r-l}(z)}{(k+r-l)!}
      -
      \sum_{j=0}^{r-l}{(-1)^j}\binom{k+j-1}{j}
      \frac{B^{\laa r \raa}_{r-j-l}(z)}{(r-j-l)!}
      \frac{B_{k+j}(\{rz\})}{(k+j)!}
      \Bigr)
      \notag\\
      &\ -
      \frac{(2\pi)^{k+r-1}}{2}
      \Bigl(
      \frac{B^{\laa r \raa}_{k+r-2}(z)}{(k+r-2)!}
      -
      \sum_{j=0}^{r-2}{(-1)^j}\binom{k+j-1}{j}
      \frac{B^{\laa r \raa}_{r-j-2}(z)}{(r-j-2)!}
      \frac{B_{k+j}(\{rz\})}{(k+j)!}
      \Bigr)
      \notag\\
      &\ -
      (2\pi)^{k+r}
      \Bigl(
      \frac{B^{\laa r \raa}_{k+r}(z)}{(k+r)!}
      -\sum_{j=0}^{r}{(-1)^j}\binom{k+j-1}{j}
      \frac{B^{\laa r \raa}_{r-j}(z)}{(r-j)!}
      \frac{B_{k+j}(\{rz\})}{(k+j)!}
      \Bigr)\in\mathbb{Q}[\pi,\varpi^4,z].\notag
\end{align}
\end{theorem}

\begin{example}
From this theorem, we give the following examples.
$$
%\sum_{m\in\mathbb{Z}\smallsetminus\{0\}}\sum_{n\in\mathbb{Z}}
\sum_{(m,n)\in \mathbb{Z}^2 \atop m\neq 0}
\frac{1}{(\sinh(m\pi))^2 (m+ni)^2}
=\frac{\varpi^4}{15\pi^2}-\frac{11}{45}\pi^2+\frac{2}{3}\pi,
$$
$$
{\Lim}_{M,N}\sum_{\substack{-M\leq m\leq M \\ -N\leq n\leq N \\ m\neq 0}}
\frac{(-1)^n}{(\sinh(m\pi))^5(m+ni)}=
-\frac{\varpi^4}{15\pi^3}+\frac{191}{945}\pi-\frac{8}{15}.
$$
\if0
where $\lim_{M,N}$ means one of:
$$
\lim_{M\to\infty}\lim_{N\to\infty}, \quad {\rm or} \quad
\lim_{N\to\infty}\lim_{M\to\infty}, \quad {\rm or} 
$$
$M,N$ tends to infinity simultaneously (that is, $M=\{M_k\},N=\{N_k\}$, and for
any $R>0$, there exists $K=K(R)$, such that $M_k\geq R, N_k\geq R$ for any
$k\geq K$). (Kronecker's sum)
\fi
\end{example}

For the case $\tau=\rho$, we first recall the properties of $G_{2j}(\rho)$ (see, for example, \cite{LM,Nes,Wa}). 
Let 
\begin{equation*}
\widetilde{\varpi} =2\int_0^1 \frac{1}{\sqrt{1-x^6}}dx
=\frac{\Gamma(1/3)^3}{2^{4/3}\pi}=2.4286506\cdots.
\end{equation*}
Then it is known that 
$G_{6k}(\rho)\in \mathbb{Q}\cdot \widetilde{\varpi}^{6k}$ $(k\in \mathbb{N})$, 
for example,
\begin{equation}
\begin{split}
& G_6(\rho)=\frac{\widetilde{\varpi}^6}{35},\ \  G_{12}(\rho)=\frac{\widetilde{\varpi}^{12}}{7007},\ \ G_{18}(\rho)=\frac{\widetilde{\varpi}^{18}}{1440257}.
\end{split}
\label{Hur-rho}
\end{equation}
Here, combining Theorems \ref{Th-4-1} and \ref{Th-4-2}, we obtain the following examples.

\begin{example}
\begin{align*}
%&\sum_{m\in\mathbb{Z}\smallsetminus\{0\}}\sum_{n\in\mathbb{Z}}
& \sum_{(m,n)\in \mathbb{Z}^2 \atop m\neq 0}
\frac{(-1)^n}{\sinh(m\pi i/\rho) (m+n\rho)^5}=\rho i\left(-\frac{\widetilde{\varpi}^6}{35\pi}+\frac{31}{2520}\pi^5-\frac{7\sqrt{3}}{540}\pi^4\right).                            
\end{align*}
\end{example}

\begin{remark}
From the above observation, we obtain the following
reciprocity formula (see \cite[Eq. (8.16)]{KMT-Forum}) for $\mathcal{G}_1(\tau)=\mathcal{G}_1^{\langle 1\rangle}(0,0,1/2;1,\tau)$:
$$
\mathcal{G}_1(\tau)+\mathcal{G}_1(-1/\tau)=\frac{\tau^2-1}{3\tau i}\pi -2.
$$
This essentially follows from the quasi-modular relation
$G_2(-1/\tau)=\tau^2 G_2(\tau)+2\pi i\tau$ (see \cite{Ko,Se}).
\end{remark}

\ 

At the end of this section,
we report an application to the theory of $q$-zeta functions.

Around 1950, Carlitz \cite{Carl} introduced a $q$-analogue of Bernoulli numbers.
In the 1980s, Koblitz \cite{Kob} introduced a $q$-analogue of the zeta-function, which
interpolates Carlitz's $q$-Bernoulli numbers. After their work, various other definitions of $q$-zeta
functions have been proposed. 
Some of them are motivated by the theory of quantum groups;
for example, Ueno and Nishizawa \cite{UN1995} introduced a $q$-analogue of Hurwitz
zeta-functions, in their study of the spectral zeta-function associated with
the quantum group $SU_q(2)$.

Here we consider the $q$-analogue of the Riemann zeta-function introduced by
Kaneko, Kurokawa and Wakayama (see \cite{KKW}):
$$
\zeta_q(s)=(1-q)^s \sum_{m=1}^{\infty}\frac{q^{m(s-1)}}{(1-q^m)^s}
$$
where $q\in\mathbb{R}$, $0<q<1$.   For any $s\in\mathbb{C}\smallsetminus\{1\}$,
we have
$$
\lim_{q\to 1}\zeta_q(s)=\zeta(s).
$$
Wakayama and Yamasaki \cite{WY} generalized the above
notion to introduce the following 2-variable $q$-zeta function:
$$
f_q(s,t)=(1-q)^s \sum_{m=1}^{\infty}\frac{q^{mt}}{(1-q^m)^s}
$$
(so, $f_q(s,s-1)=\zeta_q(s)$).
For any $(s,t)\in\mathbb{C}^2$ except $s=1$, we have
$$
\lim_{q\to 1}f_q(s,t)=\zeta(s).
$$

Here we point out that certain special values of this $f_q(s,t)$ can be  evaluated in connection
with our series involving the hyperbolic sine function.

When $q=e^{-2\pi}$, we have
$$
f_q(2k,k)=\left(\frac{1-e^{-2\pi}}{2}\right)^{2k}\sum_{m=1}^{\infty}\frac{1}
{(\sinh(m\pi))^{2k}}
$$
for any positive integer $k$. On the other hand, the following relation is a special case
of \cite[Lemma 9.1]{KMT-Forum}.

\begin{proposition}
\label{prop_p_zeta}
\begin{align*}
%\sum_{m\in\mathbb{Z}\smallsetminus\{0\}}
\sum_{m=1}^{\infty}
\frac{1}{(\sinh(m\pi))^{2k}}&=\frac{1}{\pi}
\sum_{m=1}^{\infty}\lim_{N\to\infty}\sum_{-N\leq n\leq N}
\frac{(-1)^n}{(\sinh(m\pi))^{2k-1}(m+ni)}.
\end{align*}
\end{proposition}

Therefore, applying Theorem \ref{Th-4-3}, we can evaluate the value of $f_q(2k,k)$
for $q=e^{-2\pi}$.   For example, we have
$$
f_q(2,1)=\zeta_q(2)=\frac{(1-e^{-2\pi})^2}{8}\left(\frac{1}{3}-\frac{1}{\pi}\right),
$$
$$
f_q(4,2)=\frac{(1-e^{-2\pi})^4}{32}\left(\frac{\varpi^4}{15\pi^4}-\frac{11}{45}
+\frac{2}{3\pi}\right).
$$

More generally, we can prove the same type of evaluation formula for 
$q=e^{-2\pi i/\tau}$ ($\tau\in\mathbb{H}$).   For example
$$
f_q(2,1)=\zeta_q(2)=\frac{(1-e^{-2\pi i/\rho})^2}{8}\left(\frac{1}{3}-
\frac{2}{\sqrt{3}\pi}\right)
$$
for $q=e^{-2\pi i/\rho}$.

\ 

%%%%%%%%%%%%%%%%%%%%%%%%%%%%%%%%%%%%%%
\section{Double series involving the hyperbolic sine function} \label{sec-4}
%%%%%%%%%%%%%%%%%%%%%%%%%%%%%%%%%%%%%%

In this section, we state some formulas for double series involving the hyperbolic sine function which we showed in \cite{KMT-CM}. 

Recall the special type of functional equation for the Barnes double zeta-function (by setting $(y,\tau)=(1/2,i)$ in \eqref{B-FEN}):
\begin{equation}
\sum_{m=0}^\infty\sum_{n=0}^\infty \frac{1}{(m+1/2+(n+1/2)i)^{s}}
=\frac{(2\pi)^s}{2\Gamma(s)\left( e^{\pi i s}+1\right)\left( e^{\pi i s/2}+1\right)}
%\sum_{m\in \mathbb{Z}\smallsetminus \{0\}}
\sum_{m\in \mathbb{Z} \atop m\neq 0}
\frac{(-1)^m m^{s-1}}{\sinh(m\pi)}.
\label{KMT-FE}
\end{equation}
Note that, by calculating the values at nonpositive integers on both sides of \eqref{KMT-FE}, we can obtain the Cauchy-Mellin formula \eqref{1}.

\if0
On the other hand, the values at positive integers on the left-hand side of \eqref{KMT-FE} were essentially calculated by Katayama in \cite{Ka}. In fact, he introduced $2$-division Hurwitz numbers $\{ \mathcal{E}_{4k}^{(1,1)} \in \mathbb{Q}\,|\,k\in \mathbb{N}\}$ which satisfy that 
\begin{equation}
%\sum_{m\in \mathbb{Z}}\sum_{n\in \mathbb{Z}}
\sum_{(m,n)\in \mathbb{Z}^2}
\frac{1}{(2m+1+(2n+1)i)^{4k}}=\mathcal{E}_{4k}^{(1,1)}\frac{(2\varpi)^{4k}}{(4k)!}\quad (k\in \mathbb{N}). \label{Katayama}
\end{equation}
Note that 
$$\mathcal{E}_{4}^{(1,1)}=-\frac{1}{2^5},\ \mathcal{E}_{8}^{(1,1)}=\frac{3^2}{2^9},\ \mathcal{E}_{12}^{(1,1)}=-\frac{3^4\cdot 7}{2^{13}},\ \cdots.$$
We can easily see that 
\begin{align}
& \sum_{(m,n)\in \mathbb{Z}^2}
%\sum_{m\in \mathbb{Z}}\sum_{n\in \mathbb{Z}}
\frac{1}{(2m+1+(2n+1)i)^{4k}}=\frac{4}{2^{4k}}\sum_{m=0}^\infty\sum_{n=0}^\infty \frac{1}{(m+1/2+(n+1/2)i)^{4k}}.
\label{value-4k}
\end{align}
\fi
Combining \eqref{KMT-FE} and Katayama's result for $G_{4k}^{\bf 1}(i)$ (see Section \ref{sec-3}), we obtain the following theorem.

\ 

\begin{theorem}[\cite{KMT-CM} Theorem 3.2] \label{Th-1} 
For $k\in \mathbb{N}$,
\begin{equation}
S_1(-4k+1;i)=\sum_{m=1}^\infty \frac{(-1)^m m^{4k-1}}{\sinh (m\pi)}=\frac{2^{4k-2}}{k}H_{4k}^{\bf 1} \left( \frac{\varpi}{\pi}\right)^{4k}.
\label{main2}
\end{equation}
\end{theorem}

\ 

\begin{example} \label{Ex-1}
\begin{align*}
S_1(-3;i) & = \sum_{m=1}^\infty \frac{(-1)^{m}m^3}{\sinh(m\pi)} =-\frac{1}{8}\left( \frac{\varpi}{\pi}\right)^{4},\\
S_1(-7;i) & = \sum_{m=1}^\infty \frac{(-1)^{m}m^7}{\sinh(m\pi)} =\frac{9}{16}\left( \frac{\varpi}{\pi}\right)^{8}.
\end{align*}
\end{example}

Next we consider the case $\tau=\rho$. 
\if0
Here we let
\begin{equation}
G_{2j}(\tau)=
%\sum_{m\in \mathbb{Z}}\sum_{n\in\mathbb{Z} \atop (m,n)\neq (0,0)}
\sum_{(m,n)\in \mathbb{Z}^2 \atop (m,n)\neq (0,0)}
\frac{1}{(m+n\tau)^{2j}} \label{Eisenstein}
\end{equation}
are the Eisenstein series, where $j\in \mathbb{N}$ and $\tau \in \mathbb{C}$ with $\Im \tau>0$. 
Note that even if $j=1$, we define $G_2(\tau)$ by \eqref{Eisenstein} which converges not absolutely but conditionally. 
Consider the case $\tau=\rho$. 
\fi
As well as the case $\tau=i$, we obtain the following.

\begin{theorem}[\cite{KMT-CM} Theorem 3.5]\label{Th-2} 
For $k\in \mathbb{N}$, 
\begin{equation}
S_1(-6k+1;\rho) \in \mathbb{Q}\cdot \left( \frac{\widetilde{\varpi}}{\pi}\right)^{6k}. \label{cos-rho}
\end{equation}
\end{theorem}

\begin{example} \label{Ex-2}
For example, we obtain the following:
\begin{align*}
S_1(-5;\rho) & = \sum_{m=1}^\infty \frac{(-1)^{m}m^5}{\sinh(m\pi i/\rho)} = -\frac{9}{8}\left( \frac{\widetilde{\varpi}}{\pi}\right)^{6},\\
S_1(-11;\rho) & = \sum_{m=1}^\infty \frac{(-1)^{m}m^{11}}{\sinh(m\pi i/\rho)} =\frac{30375}{16}\left( \frac{\widetilde{\varpi}}{\pi}\right)^{12}.
\end{align*}
\end{example}

\ 

Using the same method as in \cite{TsBul} and applying the above result, we obtain the following results for $\mathcal{S}_2(s;i)$ defined by \eqref{def-S2}.

\begin{theorem}[\cite{KMT-CM} Theorem 4.1]\label{Th-3-1}
For $k \in \mathbb{N}$, 
\begin{equation}
\begin{split}
\mathcal{S}_2(-4k;i) & =-\frac{4k}{\pi}S_1(-4k+1;i)=-\frac{2^{4k}}{\pi}H_{4k}^{\bf 1}\left(\frac{\varpi}{\pi}\right)^{4k}.
\end{split}
\label{eq-3-1}
\end{equation}
\end{theorem}

\begin{theorem}[\cite{KMT-CM} Theorem 4.7]\label{Th-3-2}
For $k \in \mathbb{N}_0$, 
\begin{equation}
\begin{split}
\mathcal{S}_2(4k;i) & =\frac{4k}{\pi}S_1(4k+1;i)+2\left(1-\frac{\pi}{3}\right)S_1(4k-1;i)\\
& -\frac{4}{\pi}\sum_{j=1}^{k}\zeta(4j+2){S}(4k-4j-1;i).
\end{split}
\label{e-2-1}
\end{equation}
\end{theorem}

\begin{remark} \label{Rem-3-7}
Combining \eqref{1} and \eqref{e-2-1}, we see that 
$\pi \,\mathcal{S}_2(4k;i)-4k\,S_1(4k+1;i) \in \mathbb{Q}[\pi]$ $(k\in \mathbb{N}).$ 
However it is unclear whether $S_1(4k+1;i)$ can be written in terms of $\pi$, $\varpi$ and so on, so is $S_2(4k;i)$. 
%\if0
%On the other hand, 
%by \eqref{1} and \eqref{sinh-Eisen}, we see that $\pi\, \mathcal{G}_{2p+1}(i) \in \mathbb{Q}\left[ \pi,\varpi^4\right]$ $(p \in \mathbb{N}_0)$. Therefore, by \eqref{S-2k}, we have $\pi^2 \mathfrak{S}_2(4p+2;i) \in \mathbb{Q}\left[\pi,\varpi^4\right]$ $(p\in \mathbb{N}_0)$, while we cannot evaluate $S_2(4p+2;i)$ individually. 
%\fi
\end{remark}

%In the next section, we give $\cosh$-versions of these theorems which come from a certain functional relation. Consequently this causes a different situation from the above $\sinh$-case.

\ 

%%%%%%%%%%%%%%%%%%%%%%%%%%%%%%%%%%%%%%%%%%%%%%%%%%%%%%%%%%%%%%%%%%
\section{A functional relation between $\mathcal{C}_2(s;\tau)$ and $S_1(s;\tau)$} \label{sec-5}
%%%%%%%%%%%%%%%%%%%%%%%%%%%%%%%%%%%%%%%%%%%%%%%%%%%%%%%%%%%%%%%%%%%

In this section, we consider $\mathcal{C}_2(s;\tau)$ defined by \eqref{1-3}, and prove cosh-versions of the theorems in Section \ref{sec-4}.
As we mentioned in Remark \ref{Rem-3-7}, it is unclear whether we can express $\mathcal{S}_2(4k;i)$ $(k\in \mathbb{N})$ in terms of $\pi$, $\varpi$, and so on. 
On the other hand, we can prove the following simple but non-trivial functional relation between $\mathcal{C}_2(s;\tau)$ and $S_1(s;\tau)$, which ensures some arithmetic properties of $\mathcal{C}_2(s;\tau)$ such as $\mathcal{C}_2(4k;i)\in \mathbb{Q}\cdot \pi^{4k-1}$\ $(k\in \mathbb{N})$.   This is an interesting point, because the situation is different from the hyperbolic sine case.
%, to which we can apply the results given in Section \ref{sec-2}. 

\begin{theorem} \label{Th-main-1}
Let $\tau\in \mathbb{C}$ with $\Im \tau>0$. For $s\in \mathbb{C}$, 
\begin{equation}
\begin{split}
 \mathcal{C}_2(s;\tau)=-2S_1(s-1;\tau).
\end{split}
\label{main}
\end{equation}
\end{theorem}

In order to prove this theorem, we prepare the following lemmas.

\begin{lemma} \label{L-0} 
Let $(u,v)\in\mathbb{R}^2\smallsetminus \mathbb{Z}^2$. There uniquely exists a meromorphic function $h(z)=h(z;u,v)$ in $z\in\mathbb{C}$ satisfying
\begin{equation}
  h(z)=\frac{1}{z-1/2}+O(1),\qquad h(z+1)=e^{2\pi iu}h(z),\qquad h(z+\tau/2)=e^{2\pi i v}h(z) \label{5-3}
\end{equation}
with simple poles only at $z\in\mathbb{Z}+1/2+\tau\mathbb{Z}/2$. Furthermore, when $v=0$, we have
\begin{equation}
  h(z)=h(z;u,0)=\frac{2}{u\tau}+O(1) \label{5-12}
\end{equation}
as $u\to 0$ for $z\notin \mathbb{Z}+1/2+\tau\mathbb{Z}/2$, and
the residue at $z=m+1/2$ $(m\in\mathbb{Z})$ is $e^{2\pi imu}$.
\end{lemma}

\begin{proof}
First we show the uniqueness.
 Suppose there exist $h_1(z)$ and $h_2(z)$ satisfying \eqref{5-3}. Let $H(z)=h_1(z)-h_2(z)$. Then $H(z)$ is entire and satisfies
 \begin{equation}
   \label{eq:qperiod}
   H(z+1)=e^{2\pi iu}H(z), \qquad H(z+\tau/2)=e^{2\pi i v}H(z).
 \end{equation}
 Hence $H(z)$ is a bounded entire function, namely is a constant function. 
Due to the quasi-periodicity \eqref{eq:qperiod}, this constant must be zero, because at least one of $e^{2\pi iu}$ and $e^{2\pi iv}$ is not equal to $0$.
This implies the uniqueness.

Next we show that the desired function can be given explicitly by
\begin{equation}
  h(z)=e^{2\pi i u(z-1/2)}\frac{\theta'(0;\tau/2)\theta(z-1/2+u\tau/2-v;\tau/2)}{\theta(z-1/2;\tau/2)\theta(u\tau/2-v;\tau/2)}, \label{5-4}
\end{equation}
where $\theta$ is the theta function defined by \eqref{theta-def}.
Using the properties
$$\theta(0;\tau)=0,\quad \theta(z+1;\tau)=-\theta(z;\tau),\quad
\theta(z+\tau;\tau)=-\exp(-\pi i\tau-2\pi i z)\theta(z;\tau),
$$
we can easily verify that $h(z)$ in \eqref{5-4} satisfies the conditions \eqref{5-3} with 
simple poles only at $z\in\mathbb{Z}+1/2+\tau\mathbb{Z}/2$.

The asymptotic behavior \eqref{5-12} follows from the Laurent expansion of \eqref{5-4} in $u$.   The evaluation of the residue is easy.
\end{proof}

\begin{lemma} \label{L-1} 
For $k\in \mathbb{N}$ and $q\in\mathbb{C}$ with $|q|<1$, we have
\begin{equation}
\sum_{m\in \mathbb{Z}}%\sum_{n\in \mathbb{Z}\atop m+n+1=k} 
\frac{1}{\left(q^{2m+1}+q^{-2m-1}\right)\left(q^{2(m-k)+1}+q^{-2(m-k)-1}\right)}=-\frac{k}{q^{2k}-q^{-2k}}.
\label{ee-2-1}
\end{equation}
\end{lemma}

\begin{proof}
First we let $q=e^{-\pi i/\tau}$ for any $\tau\in\mathbb{C}$ with $\Im \tau>0$, which satisfies $|q|<1$. Let $h(z)$ be the function given in Lemma \ref{L-0}. 
We define
\begin{equation}
  \begin{split}
  f(z)
  &=\frac{1}{(e^{-2\pi iz/\tau}+e^{2\pi iz/\tau})(e^{-2\pi i(z-k)/\tau}+e^{2\pi i(z-k)/\tau})}h(z)\\
  &=\frac{1}{(q^{2 z}+q^{-2 z})(q^{2 (z-k)}+q^{-2 (z-k)})}h(z).
\end{split}
\end{equation}
Then 
\begin{equation}
  f(z+\tau/2)=e^{2\pi i v}f(z).
\end{equation}
Besides the poles due to $h(z)$, the function $f(z)$ has simple poles at
$z\in\tau(\mathbb{Z}/2+1/4)$ and
$z\in\tau(\mathbb{Z}/2+1/4)+k$.

Now put $v=0$.
Let $\varepsilon$ be a sufficiently small positive number.
By considering the contour integral $C$ with vertices $\pm N\pm\tau/4+\varepsilon\tau$,
we have
%\begin{multline}
\begin{align}
\begin{split}
  \frac{1}{2\pi i}\int_{C} f(z) dz&=
  \sum_{m=-N}^{N-1}
\frac{e^{2\pi imu}}{(q^{2 (m+1/2)}+q^{-2 (m+1/2)})(q^{2 (m+1/2-k)}+q^{-2 (m+1/2-k)})}
\\
&+\Res_{z=\tau/4}f(z)
+\Res_{z=\tau/4+k}f(z).
\end{split}
\end{align}
%\end{multline}
Furthermore
\begin{align}
  \Res_{z=\tau/4}f(z)&=\frac{-\tau}{4\pi}\frac{-1}{i(q^{-2k}-q^{2k})}h(\tau/4),\\
  \Res_{z=\tau/4+k}f(z)&=\frac{-1}{i(q^{2k}-q^{-2k})}\frac{-\tau}{4\pi}e^{2\pi i ku}h(\tau/4).
\end{align}
We observe
\begin{equation}
\lim_{N\to\infty}  \frac{1}{2\pi i}\int_{C} f(z) dz=0,  
\end{equation}
because the integrals over two horizontal segments are cancelled with each other, while the integrals over two vertical segments tend to $0$.
Therefore
we obtain
\begin{multline}
    \lim_{N\to\infty}\sum_{m=-N}^{N-1}
\frac{e^{2\pi imu}}{(q^{2 (m+1/2)}+q^{-2 (m+1/2)})(q^{2 (m+1/2-k)}+q^{-2 (m+1/2-k)})}
\\
=
\frac{1}{4\pi}\frac{\tau i}{(q^{2k}-q^{-2k})}(e^{2\pi iku}-1)h(\tau/4).
\end{multline}
%Since
%\begin{equation}
%  h(z)=\frac{2}{x\tau}+O(1)
%\end{equation}
%with respect to $x$,
By \eqref{5-12}, 
we have
\begin{equation}
  \begin{split}
  \lim_{u\to 0}
\frac{1}{4\pi}\frac{\tau i}{(q^{2k}-q^{-2k})}(e^{2\pi iku}-1)h(\tau/4)
&=
\frac{1}{4\pi}\frac{\tau i}{(q^{2k}-q^{-2k})}(2\pi ik)\frac{2}{\tau}
\\
&=-
\frac{k}{(q^{2k}-q^{-2k})}.
\end{split}
\end{equation}
Thus \eqref{ee-2-1} holds for $q=e^{-\pi i/\tau}$ for any $\tau\in\mathbb{C}$ with $\Im \tau>0$. Here we note that each side of \eqref{ee-2-1} is absolutely convergent for $|q|<1$. Hence \eqref{ee-2-1} holds for all $q\in \mathbb{C}$ with $|q|<1$. This completes the proof.
\end{proof}

\begin{proof}[Proof of Theorem \ref{Th-main-1}] 
For $t \in \mathbb{C}$ with $\Re t>0$, let 
\begin{align}\label{ee-2-2}
F(t;\tau)=\sum_{(m,n)\in \mathbb{Z}^2 \atop m+n+1>0}
%\sum_{m\in \mathbb{Z}}\sum_{n\in \mathbb{Z}\atop m+n+1>0}
\frac{(-1)^{m+n+1}e^{-(m+n+1)t}}{\cosh((m+1/2)\pi i/\tau)\cosh((n+1/2)\pi i/\tau)}.
\end{align}
Since the right-hand side is equal to
\begin{align*}
%\sum_{m\in \mathbb{Z}}\sum_{k\in \mathbb{N}}
\sum_{(m,n)\in \mathbb{Z}^2}
\frac{(-1)^{k}e^{-kt}}{\cosh((m+1/2)\pi i/\tau)\cosh(-(m-k+1/2)\pi i/\tau)},
\end{align*}
by Lemma \ref{L-1} with $q=e^{-\pi i/2\tau}$, we have
\begin{equation}
F(t;\tau)=2\sum_{k=1}^\infty \frac{(-1)^{k}ke^{-kt}}{\sinh(k\pi i/\tau)}. \label{ee-2-3}
\end{equation}
Note that $F(t;\tau)$ is holomorphic for $t\in \mathbb{C}$ with $\Re t>0$. 
%We again consider the contour integral. Let $\Upsilon$ be the path which consists of the positive real axis $[\varepsilon, \infty]$ (top side), a circle $C_{\varepsilon}$ around $0$ of radius $\varepsilon$, and the positive real axis $[\varepsilon, \infty]$ (bottom side), where $0< \varepsilon<\rho$. Note that we interpret $t^s$ as $\exp(s\log t)$, where the imaginary part of $\log t$ varies from $0$ (on the top side of the real axis) to $2\pi$ (on the bottom side). 

For $s\in \mathbb{C}$ with $\Re s>1$, let 
\begin{equation}
\begin{split}
I(s;\tau)%& =\int_\Upsilon F(-t;\tau)t^{s-1} dt \\
&= \int_{0}^\infty F(t;\tau)t^{s-1}dt. % + \int_{C_\varepsilon} F(-t;\tau)t^{s-1}dt.
\end{split}
\label{ee-2-4}
\end{equation}
This is holomorphic for $s \in \mathbb{C}$ with $\Re s>1$. 
Using \eqref{ee-2-2}, we have 
\begin{align*}
I(s;\tau) %=\left( e^{2\pi is}-1\right) \int_{0}^\infty F(-t;\tau)t^{s-1}dt \\
& =\int_{0}^\infty 
%\sum_{m\in \mathbb{Z}}\sum_{n\in \mathbb{Z}\atop m+n+1>0}
\sum_{(m,n)\in \mathbb{Z}^2 \atop m+n+1>0}
\frac{(-1)^{m+n+1}e^{-(m+n+1)t}}{\cosh((m+1/2)\pi/\tau)\cosh((n+1/2)\pi/\tau)}t^{s-1}dt\\
& =\Gamma(s)
%\sum_{m\in \mathbb{Z}}\sum_{n\in \mathbb{Z}\atop m+n+1>0}
\sum_{(m,n)\in \mathbb{Z}^2 \atop m+n+1>0}
\frac{(-1)^{m+n+1}}{\cosh((m+1/2)\pi i/\tau)\cosh((n+1/2)\pi i/\tau)(m+n+1)^s}\\
& =-\Gamma(s)\mathcal{C}_2(s;\tau).
\end{align*}
Also, by using \eqref{ee-2-3}, we similarly obtain
\begin{align*}
I(s;\tau)& =2\Gamma(s)S_1(s-1;\tau).
%\sum_{l=1}^\infty \frac{(-1)^{l}}{\sinh(l\pi i/\tau)l^{s-1}}.
\end{align*}
Therefore we see that \eqref{main} holds for $\Re s>1$. Since $\mathcal{C}_2(s;\tau)$ and $S_1(s-1;\tau)$ are entire, \eqref{main} holds for all $s\in \mathbb{C}$. This completes the proof.
\end{proof}

From Theorem \ref{Th-1}, we have the following.

\begin{corollary}\label{Cor-4-3}
For $k\in \mathbb{N}$, 
\begin{equation}
\mathcal{C}_2(-4k+2;i)=-\frac{2^{4k-1}}{k}H_{4k}^{\bf 1} \left( \frac{\varpi}{\pi}\right)^{4k}. \label{eq-4-3}
\end{equation}
\end{corollary}

\begin{example} 
\label{exam-2}
\begin{align*}
\mathcal{C}_2(-2;i) & = 
%\sum_{m\in \mathbb{Z}} \sum_{n\in \mathbb{Z} \atop m+n+1>0}
\sum_{(m,n)\in \mathbb{Z}^2 \atop m+n+1>0}
\frac{(-1)^{m+n}(m+n+1)^{2}}{\cosh((m+1/2)\pi)\cosh((n+1/2)\pi)} =\frac{1}{4}\left( \frac{\varpi}{\pi}\right)^{4},\\
\mathcal{C}_2(-6;i) & = 
%\sum_{m\in \mathbb{Z}} \sum_{n\in \mathbb{Z} \atop m+n+1>0}
\sum_{(m,n)\in \mathbb{Z}^2 \atop m+n+1>0}
\frac{(-1)^{m+n}(m+n+1)^{6}}{\cosh((m+1/2)\pi/2)\cosh((n+1/2)\pi)} =-\frac{9}{8}\left( \frac{\varpi}{\pi}\right)^{8},\\
\mathcal{C}_2(-10;i) & =
%\sum_{m\in \mathbb{Z}} \sum_{n\in \mathbb{Z} \atop m+n+1>0}
\sum_{(m,n)\in \mathbb{Z}^2 \atop m+n+1>0}
\frac{(-1)^{m+n}(m+n+1)^{10}}{\cosh((m+1/2)\pi/2)\cosh((n+1/2)\pi)} = \frac{189}{4}\left( \frac{\varpi}{\pi}\right)^{12}.
\end{align*}
\end{example}

\begin{example} 
\label{exam-3} 
Combining Theorems \ref{Th-2} and \ref{Th-main-1}, we obtain
\begin{equation}
\mathcal{C}_2(-6k+2;\rho)\in \mathbb{Q}\cdot \left( \frac{\widetilde{\varpi}}{\pi}\right)^{6k}. \label{c2-rho}
\end{equation}
In fact, by Example \ref{Ex-2}, we can explicitly give 
\begin{align*}
\mathcal{C}_2(-4;\rho) & = \sum_{(m,n)\in \mathbb{Z}^2 \atop m+n+1>0}\frac{(-1)^{m+n}(m+n+1)^{4}}{\cosh((m+1/2)\pi i/\rho)\cosh((n+1/2)\pi i/\rho)} = \frac{9}{4}\left( \frac{\widetilde{\varpi}}{\pi}\right)^{6},\\
\mathcal{C}_2(-10;\rho) & = \sum_{(m,n)\in \mathbb{Z}^2 \atop m+n+1>0}\frac{(-1)^{m+n}(m+n+1)^{10}}{\cosh((m+1/2)\pi i/\rho)\cosh((n+1/2)\pi i/\rho)} =-\frac{30375}{8}\left( \frac{\widetilde{\varpi}}{\pi}\right)^{12},\\
\mathcal{C}_2(-16;\rho) & = \sum_{(m,n)\in \mathbb{Z}^2 \atop m+n+1>0}\frac{(-1)^{m+n}(m+n+1)^{16}}{\cosh((m+1/2)\pi i/\rho)\cosh((n+1/2)\pi i/\rho)}  =\frac{658560375}{4}\left( \frac{\widetilde{\varpi}}{\pi}\right)^{18}.
\end{align*}
\end{example}

Combining \eqref{main} and \eqref{1} with $s=2k$ $(k\in \mathbb{N})$, we easily obtain 
the following result which can be regarded as a double analogue of \eqref{1-2}.

\begin{corollary} \label{C-1} \ For $k\in \mathbb{N}_0$, we have $\mathcal{C}_2(-4k-4;i)=0$ and 
\begin{equation}
\begin{split}
\mathcal{C}_2(4k;i) & =\sum_{(m,n)\in \mathbb{Z}^2 \atop m+n+1>0}\frac{(-1)^{m+n}}{\cosh((m+1/2)\pi)\cosh((n+1/2)\pi)(m+n+1)^{4k}} \\
& ={(2\pi)^{4k-1}}\sum_{j=0}^{2k}(-1)^{j}\frac{B_{2j}\left( \frac{1}{2}\right)}{(2j)!}\frac{B_{4k-2j}\left( \frac{1}{2}\right)}{(4k-2j)!}.
\end{split}
\label{1-4}
\end{equation}
\end{corollary}

\begin{example} \label{Exam-1}
The following formulas can be derived from \eqref{1-4}:
\begin{align*}
& \mathcal{C}_2( 4;i )=\frac{1}{1440}\pi^3, \\
& \mathcal{C}_2( 8;i )=\frac{13}{29030400}\pi^7,\\
& \mathcal{C}_2( 12;i )=\frac{4009}{13948526592000}\pi^{11},\\
& \mathcal{C}_2( 16;i )=\frac{13739}{74500348575744000}\pi^{15}.\\
\end{align*}
\end{example}

\bigskip

\bibliographystyle{amsplain}

\begin{thebibliography}{100}


\bibitem{Be2}
{B. C. Berndt}, \emph{Ramanujan's Notebooks, part II}, Springer-Verlag, New-York, 1989.

\bibitem{Carl}
L. Carlitz, $q$-Bernoulli and Eulerian numbers, {Trans.\ Amer.\ Math.\ Soc.} \textbf{76} (1954), 332--350.

\bibitem{Ca}
{A. L. Cauchy},
\emph{Exercices de Math{\'e}matiques}, Paris, 1827;
\emph{Oeuvres Completes D'Augustin Cauchy},
S\'erie II, t.\ VII, Gauthier-Villars, Paris, 1889.


%\bibitem {Di}
%{K. Dilcher}, \emph{Zeros of Bernoulli, Generalized Bernoulli and Euler Polynomials}, 
%{Mem. Amer. Math. Soc.} \textbf{386}, 1988.


\bibitem{Hecke1927}
E. Hecke, Theorie der Eisensteinschen Reihen h\"oherer Stufe und ihre Anwendung auf Funktionentheorie und Arithmetik, Abh.\ Math.\ Sem.\ Univ.\ Hamburg  {\bf 5}  (1927), 199--224. 


\bibitem{Her}
{G. Herglotz}, {\"U}ber die Entwicklungskoeffizienten der Weierstrassschen $\wp$-Funktion, Berichte {\"u}ber der Verhandlungen der S\"achsischen Akademie der Wissenschaften zu Leibzig, Math.-phys. Klasse {\bf 74} (1922), 269--289.

\bibitem{Hur99}
{A. Hurwitz}, {\"U}ber die Entwicklungskoeffizienten der lemniskatischen
   Funktionen, {Math.\ Ann.} \textbf{51} (1899), 342--373.

\bibitem{KKW}
{M. Kaneko, N. Kurokawa and M. Wakayama}, A variation of Euler's approach to values of the Riemann zeta function, {Kyushu J. Math.} {\bf 57} (2003), 175--192.

\bibitem{Katayama}
{K. Katayama}, 
 On the values of Eisenstein series,
 {Tokyo J. Math.} \textbf{1} (1978), 157--188.

\bibitem{Kob}
{N. Koblitz}, {On Carlitz's $q$-Bernoulli numbers}, {J. Number Theory} \textbf{14} (1982), 332--339.

\bibitem{Ko}
{N. Koblitz}, \emph{Introduction to Elliptic Curves and Modular Forms, Second edition}, Graduate Texts in Mathematics, No. 97, Springer-Verlag, New York, 1993.


\bibitem{KMT-DB}
{Y. Komori, K. Matsumoto and H. Tsumura}, Functional equations and functional relations for Euler double zeta-function and its generalization of Eisenstein type, {Publ.\ Math.\ Debrecen} {\bf 77} (2010), 15--31. 


\bibitem{KMT-CM}
{Y. Komori, K. Matsumoto and H. Tsumura}, Evaluation formulas of Cauchy-Mellin type for certain series involving hyperbolic functions, Comment.\ Math.\ Univ.\ St. Pauli {\bf 60} (2011), 127--142.


\bibitem{KMT-CRB}
{Y. Komori, K. Matsumoto and H. Tsumura}, Barnes multiple zeta-functions, Ramanujan's formula, and relevant series involving hyperbolic functions,  J. Ramanujan Math.\ Soc. {\bf 28} (2013), 49-69.

\bibitem{KMT-Forum}
{Y. Komori, K. Matsumoto and H. Tsumura}, Hyperbolic-sine analogues of Eisenstein series, generalized Hurwitz numbers, and $q$-zeta functions, Forum Math. {\bf 26} (2014), 1071--1115. 


\bibitem{Kro}
{L. Kronecker}, Bemerkungen {\" u}ber die Darstellung von Reihen durch Integrale, {J. Reine Angew. Math.} {\bf 105} (1889), 157--159, 345--354, also in \emph{Leopold Kronecker's Werke V}, Herausgegeben auf Veranlassung der Preuss.\ Akad.\ Wiss., reprinted by Chelsea, 1968.  
 
\bibitem{LM}
{F. Lemmermeyer}, \textit{Reciprocity Laws: From Euler to Eisenstein}, Springer, 2000.

\bibitem{MT2006}
{K. Matsumoto and H. Tsumura}, Generalized multiple Dirichlet series and generalized multiple polylogarithms, Acta Arith. {\bf 124} (2006), 139--158. 

\bibitem{Me1}
{Hj.\ Mellin}, 
{Eine Formel f\"ur den Logarithmus transcendenter Funktionen 
von endlichem Geschlecht}, Acta Soc.\ Sci.\ Fennicae \textbf{29}, n.4 (1902), 49pp. 

\bibitem{Me2}
{Hj.\ Mellin}, 
{Eine Formel f\"ur den Logarithmus transcendenter Funktionen 
von endlichem Geschlecht}, 
Acta Math. \textbf{25}, n.1 (1902), 165--183. 

\bibitem{Nes}
Yu. V. Nesterenko, Modular functions and transcendence questions (Russian), Mat.\ Sb. {\bf 187} (1996), 65--96; translation in Sb.\ Math. {\bf 187} (1996), 1319--1348.

\bibitem{Nor0}
N. E. N{\"o}rlund, M\'emoire sur les polynomes de Bernoulli, {Acta Math.} \textbf{43} (1922), 121--196.

\bibitem{Nor}
N. E. N{\"o}rlund, \emph{Vorlesungen {\"u}ber Differenzenrechnung}, Verlag von Julius Springer, Berlin, 1924.

\bibitem{Se}
{J.-P. Serre}, \emph{A Course in Arithmetic}, Graduate Texts in Mathematics, No.\ 7, Springer-Verlag, New York-Heidelberg, 1973.

\bibitem{Shibukawa}
G. Shibukawa, Bilateral zeta functions and their applications, Kyushu J. Math.\ {\bf 67} (2013), 429--451.


\bibitem{TsAust}
{H. Tsumura}, Certain functional relations for the double harmonic series related to the double Euler numbers, J. Austral.\ Math.\ Soc., Ser.\ A. {\bf 79} (2005), 319--333. 

\bibitem{TsBul}
{H. Tsumura}, On certain analogues of Eisenstein series and their evaluation formulas of Hurwitz type, {Bull.\ London Math. Soc.} \textbf{40} (2008), 85--93.

\bibitem{TsResult}
{H. Tsumura}, Analogues of the Hurwitz formulas for level $2$ Eisenstein series, Results Math. {\bf 58} (2010), 365--378. 

\bibitem{UN1995}
K. Ueno and  M. Nishizawa, Quantum groups and zeta-functions, in \emph{Quantum groups.  Formalism and Applications} (Karpacz, 1994),
J. Lukierski et al. (eds.), 115--126, PWN, Warsaw, 1995. 

\bibitem{WY}
{M. Wakayama and Y. Yamasaki}, Integral representations of $q$-analogues of the Hurwitz zeta function, {Monatsh.\ Math.} {\bf 149} (2006), 141--154.

\bibitem{Wa}
{G. N. Watson},
{Theorems stated by Ramanujan II}, J. London Math.\ Soc. \textbf{3} (1928), 216--225.

\end{thebibliography}

\end{document}